\documentclass[12pt]{amsart}
\usepackage{graphicx}
\usepackage{amssymb}
\usepackage{verbatim}
\usepackage{epsfig}

\newtheorem{thm}{Theorem}

\newtheorem{cor}[thm]{Corollary}
\theoremstyle{definition}

\newtheorem{rem}[thm]{Remark}

\newcommand{\bn}{\par\bigskip\noindent}
\newcommand{\pars}{\par\smallskip}

\newcommand{\adresse}{\par\bigskip \small\rm
Cimpri\v c, Jaka\par
Faculty of Mathematics and Physics\par
University of Ljubljana\par
Jadranska 21, SI-1000 Ljubljana, Slovenija\par
email: cimpric@fmf.uni-lj.si
\pars \pars \pars
Kuhlmann, Salma\par
Department of Mathematics and Statistics\par
University of Saskatchewan\par
Saskatoon, SK S7N 5E6, Canada\par
email: skuhlman@math.usask.ca
\pars \pars \pars
Marshall, Murray\par
Department of Mathematics and Statistics\par
University of Saskatchewan\par
Saskatoon, SK S7N 5E6, Canada\par
email: marshall@math.usask.ca}

\def\psd{\mathrm{Psd}}
\def\R{\mathbb{R}}
\def\rx{\R[\mathbf{x}]}

\begin{document}

\title{Positivity in power series rings}
\author {J.\ Cimpri\v c,
S.\ Kuhlmann, M. Marshall}
\thanks {%The first author was supported by a Nato Science Fellowship.
The second and third authors were partially supported by NSERC
Discovery grants.}
%\date{31.03.2008.}
\begin{abstract}
We extend and generalize %the
results of Scheiderer (2006) on the representation of polynomials
nonnegative on two-dimensional basic closed semialgebraic sets. Our
extension covers some situations where the defining polynomials do
not satisfy the transversality condition. Such situations arise
naturally when one considers semialgebraic sets invariant under finite
group actions.
\end{abstract}
\maketitle

\section{Introduction}

Let $\rx:=\R[x_1, \cdots, x_n]$ be the ring of polynomials in $n$ variables with real coefficients. A \textbf{preordering} of a general ring $A$ (commutative with $1$) is a subsemiring of $A$ which contains the squares. In other words, a preordering of $A$ is a subset of $A$ which contains all $f^2$, $f\in A$, and is closed under addition and multiplication.
%Write $\sos$ for the set of all sums of squares of polynomials from $\rx$. Note that $\sos$ is a subsemiring of $\rx$. Any subsemiring of $\rx$ containing $\sos$ is called a \textbf{preordering}. In other words, preorderings of $\rx$ are the subsets of $\rx$ that contain $\sos$ and are closed for addition and multiplication.
For a finite subset $S=\{g_1, ..., g_s\}$ of $\rx$, we write $T_S$ for the preordering of $\rx$ generated by $S$,
and $K_S$ for the set of all $x \in \R^n$ satisfying $g_1(x) \ge 0, \dots, g_s(x) \ge 0$ (the basic closed semialgebraic set defined by $S$). Note that $K_S$ is
uniquely determined by $T_S$, but typically $T_S$ is not uniquely determined by $K_S$.
For a subset $K$ of $\R^n$, we write $\psd(K)$ for the set of all elements of $\rx$ that are nonnegative on $K$.
We always have that $T_S \subseteq \psd(K_S)$. The preordering $T_S$ is said to be \textbf{saturated} if
$T_S=\psd(K_S)$.

In this paper we investigate what geometric properties of $S$ imply that $T_S$
is saturated. This line of investigation has been pursued by Scheiderer in a series of papers.
In \cite{S1}, Scheiderer showed that $T_S$ is never saturated if $\dim(K_S)\geq 3$.
The case $\dim(K_S)\leq 1$ is fairly well understood; see \cite{km}, \cite{kms}, \cite{pl}, \cite{S2}.
We focus here on the 2-dimensional case, more precisely, on the affine 2-dimensional case, i.e., $n=\dim(K_S) = 2$.
%So, the interesting case is $\dim(K_S)= 2$.
%Moreover, we assume that $n=2$ (so $K_S$ has nonempty interior.)

We consider only the compact case. In the non-compact case little is known; %In fact, for a long time, not a single non-compact 2-dimensional example was known;
see \cite[Open Problem 6]{km} and \cite[Remark 3.16]{S3}. By \cite[Remark 6.7]{S1},
$T_S$ is not saturated if $K_S$ contains a two-dimensional cone.
In the compact case, we have the following result of Scheiderer \cite[Cor. 3.3]{S3}:

\begin{thm} \label{sat3}
Let $S=\{g_1, ..., g_s\}$ be irreducible polynomials in $\R[x,y]$,
let $C_i$ be the plane affine curve $g_i=0$ ($i=1, ..., s$). Assume:
\begin{enumerate}
\item $K_S$ is compact
\item $C_i$ has no real singular points ($i=1, ..., s$)
\item the real points of intersection of any two of the $C_i$ are
transversal, and no three of the $C_i$ intersect in a real point.
\end{enumerate}
Then $T_S$ is saturated.
\end{thm}

The main goal of this paper is to show that saturation
holds in certain other compact cases as well, e.g., if
$S = \{ x,1-x, y, x^2-y\}$ or $S =\{ 1+x,1-x, y, x^2-y\}$.
In these examples, the boundary curves $y=0$ and $y=x^2$ share a
common tangent at the origin, so Theorem \ref{sat3} does not apply. The fact that saturation holds in these examples is a consequence of our main result, Corollary \ref{c7m}, which is an extension of Theorem \ref{sat3}. %These two cases are instances of
%assertions (3) and (4) of Corollary \ref{c7m} below, which, in some sense, is our main result.
%It is a generalization of Theorem \ref{sat3}.

Our original motivation comes from examples which
arise naturally while studying semialgebraic sets $K_{S'}$ described
by a set $S'$ of polynomials invariant under an action of a finite
group $G$. The corresponding preordering $T_{S'}$ will typically not
be saturated but it can still be ``saturated for invariant polynomials''
(we refer to this as ``$G$-saturation").
The orbit map $\pi$ (see \cite{CKS}) relates the $G$-saturation of
$T_{S'}$ to the saturation of certain preordering $T_{\tilde{S'}}$
corresponding to $\pi(K_{S'})=K_{\tilde{S'}}$. In many cases,
the latter follows from our Corollary \ref{c7m}.
An example is given in Section 3.

At the same time, Corollary \ref{c7m} does not cover all interesting cases; in the
Concluding Remarks, we consider some of the remaining cases.

%---------------------------------------------------------------------

\section{Saturation in dimension two} \label{mmnote}
We focus on the case of a compact basic closed
semialgebraic set. In \cite[Cor. 3.17]{S2}, Scheiderer proves a
useful `local-global' criterion, extending \cite[Cor. 3]{Sc1}, for deciding when a polynomial
non-negative on a compact basic closed semialgebraic set lies in the
associated preordering of the polynomial ring:

\begin{thm} \label{t1m}
Suppose $f,g_1,\dots,g_s \in \rx$,
the subset $K$ of $\mathbb{R}^n$ defined by the inequalities $g_i\ge 0$,
$i=1,\dots,s$ is compact, $f\ge0$ on $K$, and $f$ has just finitely
many zeros in $K$. Then the following are equivalent:
\begin{enumerate}
\item $f$ lies in the preordering of $\rx$ generated by $g_1,\dots,g_s$.
\item For each zero $p$ of $f$ in $K$, $f$ lies in the preordering of the completion of
$\rx$ at $p$  generated by $g_1,\dots,g_s$.
\end{enumerate}
\end{thm}

In the two--dimensional case this allows one to show that certain
finitely generated preorderings are saturated; see \cite{S3}.
For example, Theorem \ref{sat3} can be obtained  by combining
Theorem \ref{t1m} with the following result for power series rings,
using the Transfer Principle:

\begin{thm} \label{t2m}
Suppose $f \in \mathbb{R}[[x,y]]$.
\begin{enumerate}
\item If $f\ge 0$ at each ordering of $\R((x,y))$ then $f$ is a sum
of squares in $\R[[x,y]]$.
\item If $f\ge 0$ at each ordering of $\R((x,y))$ satisfying $x> 0$
then $f$ lies in the preordering of $\R[[x,y]]$ generated by $x$.
\item If $f \ge 0$ at each ordering of $\R((x,y))$ satisfying $x>0$
and $y > 0$ then $f$ lies in the preordering of $\R[[x,y]]$
generated by $x$ and $y$.
\end{enumerate}
\end{thm}

\begin{proof} (1) is well-known. It can be proved using a modification
of the analytic argument given in \cite[Lem. 7a]{BR}. The proof
shows, in fact, that $f$ is a sum of two squares.  See \cite[Th.
1.6.3]{mar} for more details. (2) (resp., (3)) follows immediately from
(1) by going to the extension ring $\R[[\sqrt{x},y]]$ (resp.,
to the extension ring $\R[[\sqrt{x}, \sqrt{y}]]$). E.g., to
prove (2), apply (1) to $\R[[\sqrt{x},y]]$ to deduce $f=\sum
f_i^2$, $f_i \in \R[[\sqrt{x},y]] $. Decomposing $f_i=
f_{i1}+f_{i2}\sqrt{x}$, $f_{ij} \in \R[[x,y]]$, and expanding,
yields $f= \sum f_{i1}^2+\sum f_{i2}^2x$.
\end{proof}

We will prove the following extension of Theorem \ref{t2m}.

\begin{thm} \label{t3m}
Suppose $f \in \R[[x,y]]$ and $n$ is a positive integer.
\begin{enumerate}
\item If $f \ge 0$ at each ordering of $\R((x,y))$ satisfying $y >0$
and $x^{2n}-y > 0$ then $f$ lies in the preordering of
$\mathbb{R}[[x,y]]$ generated by $y$ and $x^{2n}-y$.
\item If $f \ge 0$ at each ordering of $\mathbb{R}((x,y))$ satisfying $x>0$,
$y > 0$ and $x^n-y > 0$  then $f$ lies in the preordering of
$\mathbb{R}[[x,y]]$ generated by $x$, $y$ and $x^n-y$.
\end{enumerate}
\end{thm}

\begin{rem} \label{rm5} Suppose $n$ is odd, $n\ge 3$. Then:
\begin{enumerate}
\item[(i)] For every ordering of $\R[[x,y]]$, $y\ge0$ and $x^n -y \ge 0$
$\Rightarrow$ $x \ge 0$, but $x$ is not in the preordering of $\R[[x,y]]$
generated by $y$ and $x^n-y$. This shows that an obvious attempt to strengthen
Theorem \ref{t3m} fails.
\item[(ii)] Similarly, for every ordering of $\R[[x,y]]$, $x^n - y^2 \ge 0$
$\Rightarrow$ $x\ge 0$, but $x$ is not in the preordering of $\R[[x,y]]$
generated by $x^n-y^2$.
\end{enumerate}
Note: Going to the extension ring $\R[[x,\sqrt{y}]]$, we see that
assertions (i) and (ii) are essentially equivalent.
\end{rem}
%At the same time, there is no claim that Theorem \ref{t3m} is the
%end of the story. It is conceivable that other similar results might
%exist. See the Concluding Remarks (Remark \ref{rm10}) for a
%systematic discussion.

We postpone the proof of Theorem \ref{t3m} to Section \ref{proof}.
For now we only explain how Theorems \ref{t1m}, \ref{t2m} and \ref{t3m} can
be combined to yield the promised extension of Theorem \ref{sat3}:

\begin{cor} \label{c7m} Let $S=\{g_1,...,g_s\}$ be irreducible polynomials in $\R[x,y]$.
Suppose that $K=K_S \subseteq \R^2$ is compact, and, for each boundary point $p$ of
$K$, either
\begin{enumerate}
\item there exists $i$ such that $p$ is a non-singular zero of $g_i$, and
$K$ is defined locally at $p$ by the single inequality $g_i\ge 0$; or
\item there exists $i,j$ such that $p$ is a non-singular zero of $g_i$
and $g_j$, $g_i$ and $g_j$ meet transversally at $p$,
and $K$ is defined locally at $p$ by $g_i\ge0$, $g_j\ge 0$; or
\item there exists $i,j$ such that $p$ is a non-singular zero of $g_i$ and $g_j$,
$g_i$ and $g_j$ share a common tangent at $p$ but do not cross each other
at $p$, and $K$ is described locally at $p$ as the region between $g_i =0$
and $g_j=0$; or
\item there exists $i,j,k$ such that $p$ is a non-singular zero of $g_i$, $g_j$
and $g_k$, $g_i$ and $g_j$ share a common tangent at $p$, $g_i$ and $g_k$
meet transversally at $p$, and $K$ is described locally at $p$ as the part
of the region between $g_i=0$ and $g_j=0$ defined by $g_k\ge 0$.
\end{enumerate}
Then the preordering of $\R[x,y]$ generated by $g_1,\dots,g_s$ is saturated.
\end{cor}

\begin{proof}
Let $T$ denote the preordering of $\R[x,y]$ generated by
$g_1,\dots,g_s$. We wish to show that $f\in \R[x,y]$, $f\ge 0$
on $K$ $\Rightarrow$ $f\in T$. We may assume $K\ne \emptyset$, $f\ne
0$. The hypothesis implies, in particular, that $K$ is the closure
of its interior. This allows us to reduce further to the case where
$f$ is square-free and $g_i \nmid f$ for each $i$. In this
situation, $f$ has only finitely many zeros in $K$, so Theorem
\ref{t1m} applies, i.e., to show $f\in T$, it suffices to show that,
for each zero $p$ of $f$ in $K$, $f$ lies in the preordering of the
completion of $\R[x,y]$ at $p$ generated by $g_1,\dots,g_s$.
If $p$ is an interior point of $K$ this follows from Theorem
\ref{t2m}(1). If $p$ is a boundary point of $K$ satisfying (1)
(resp., (2), resp., (3), resp., (4)) then it follows from Theorem
\ref{t2m}(2) (resp., Theorem \ref{t2m}(3), resp., Theorem
\ref{t3m}(1), resp., Theorem \ref{t3m}(2)). We use the Transfer Principle
and apply Theorems \ref{t2m} and \ref{t3m} with $x = \overline{x}$,
$y =\overline{y}$, where $\overline{x}, \overline{y}$ are suitably
chosen local parameters at $p$. If $p$ is an interior point of $K$
we choose $\overline{x}=x-a$, $\overline{y}=y-b$ where $p= (a,b)$.
In case (1), we choose local parameters $\overline{x}$,
$\overline{y}$ with $\overline{x}=g_i$. In case (2), we choose local
parameters $\overline{x}$, $\overline{y}$ with $\overline{x}=g_i$,
$\overline{y}=g_j$. In case (3), choose local parameters
$\overline{x}$, $g_i$. By the Preparation Theorem \cite[Cor. 1, p.
145]{ZS}, $hg_j=g_i+\overline{x}^nk$ for some unit $h$, some $n\ge
1$ and some unit $k\in \R[[\overline{x}]]$. Then
$sg_i+tg_j=\overline{x}^n$ where $s= -\frac{1}{k}$ and
$t=\frac{h}{k}$. By the geometry of the situation, the units $s,t$
are positive units and $n$ is even. Take $\overline{y} = sg_i$, so
$\overline{x}^n-\overline{y} = tg_j$, and apply Theorem
\ref{t3m}(1). In case (4) choose local parameters $\overline{x}$,
$g_i$ with $\overline{x}=g_k$. As before, this yields
$sg_i+tg_j=\overline{x}^n$ for some units $s,t$ and some $n \ge 1$.
By the geometry of the situation, $s,t$ are positive units. Take
$\overline{y} = sg_i$, so $\overline{x}^n-\overline{y} = tg_j$, and
apply Theorem \ref{t3m}(2).
\end{proof}

%--------------------------------------------------------------------------------------------------------
\section{Application to equivariant saturated preorderings}
\label{jaka}

If $S=\{1 - x, 1 + x, 1 - y, 1 + y\}$ and $S'=\{2-x^2-y^2,(1-x^2)(1-y^2)\}$
then $K_S=K_{S'}$ is the unit square. Note that $T_S$ is saturated, by Theorem \ref{sat3}.
On the other hand, it can be easily verified that $1-x \not\in T_{S'}$,
hence $T_{S'}$ is not saturated.

Let $G=\langle a,b \vert a^4=b^2=(ab)^2=1\rangle$ be the fourth
dihedral group acting on $\R^2$ and $\R[x,y]$ in a ``standard
way". For every $G$-invariant subset $M$ of $\R[x,y]$ write
$M^G=\{m \in M \vert \forall g \in G \colon m^g=m\}$. We would
like to show that $T_{S'}$ is \textbf{$G$-saturated}, i.e.
$\psd(K_S)^G \subseteq T_{S'}$ or equivalently, $\psd(K_S)^G =
(T_{S'})^G$.

Clearly,  $\R[x,y]^G$ is an $\R$-algebra containing
\[
u(x,y)=x^2+y^2 \text{ and } v(x,y)=x^2 y^2.
\]
It can be shown that $u(x,y)$ and $v(x,y)$ are algebraically independent and that
they generate $\R[x,y]^G$. Hence, the mapping
\[
\tilde{\pi} \colon \R[u,v] \to \R[x,y]^G, \quad \tilde{\pi}(f)(u,v) = f(u(x,y),v(x,y))
\]
is an isomorphism. On the other hand, the mapping
\[
\pi \colon \R^2 \to \R^2, \quad \pi(x,y) = (u(x,y),v(x,y))
\]
is not onto. It is easy to see that $\pi(\R^2)=K_{\{u,v,u^2-4v\}}.$
The mapping $\pi$ is not one-to-one either. It can be shown that two
points have the same image if and only if they lie in the same
$G$-orbit. (We call $\pi$ the \textbf{orbit map} \rm and $\pi(\R^2)$
the \textbf{orbit space}.)

The set $\Delta=\{(x,y) \vert 0 \le y \le x \le 1\}$ (picture on the left) contains exactly
one point from each orbit of $K_S$.
\begin{center}
\includegraphics{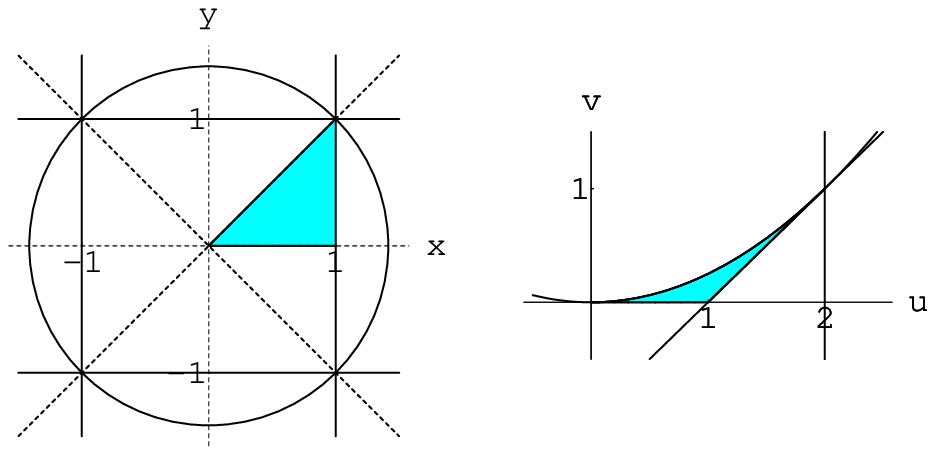}
\end{center}
Now we can compute $\pi(K_S)=\pi(\Delta)$
(picture on the right) by either parametrizing the boundary of $\Delta$ or the following way:
\[\begin{array}{c}
\pi(K_S)=\pi(K_{S'})=K_{\tilde{\pi}^{-1}(S')} \cap \pi(\R^2)
= \\ =
K_{\{2-u,1-u+v\}} \cap K_{\{u,v,u^2-4v\}}=K_{\{2-u,1-u+v,u,v,u^2-4v\}}.
\end{array}\]
By Corollary \ref{c7m}, the preordering  $T_{\{2-u,1-u+v,u,v,u^2-4v\}}$ is saturated. Hence
\[
\psd(K_S)^G = \tilde{\pi}(\psd(\pi(K_S))) \subseteq \tilde{\pi}(T_{\{2-u,1-u+v,u,v,u^2-4v\}}) \subseteq T_{S'}.
\]

\begin{comment}
Setup from Example \ref{square}. We claim that
$1-X \not\in T_{\{2-X^2-Y^2,(1-X^2)(1-Y^2)\}}=T_{S'}$.
Assume on the contrary that
\[\begin{array}{lll}
1-X & = & h_1(X,Y)+h_2(X,Y)(2-X^2-Y^2)+ \\
      &   & +[h_3(X,Y) + h_4(X,Y)(2-X^2-Y^2)](1-X^2)(1-Y^2)
\end{array}\]
where $h_1,h_2,h_3,h_4$ are sums of squares.

Replacing $Y$ by $1$ we get
$$1-X = h_1(X,1)+h_2(X,1)(1-X^2).$$
It follows that $h_1(1,1) = 0$,
hence $h_1(X,1) = (X-1)^2 p(X)$.
Cancelling $1-X$, we get
$1 = (X-1)p(X)+h_2(X,1)(X+1)$,
which implies $h_2(1,1) = \frac12$.
On the other hand replacing $X$ by $1$ we get
$$0  =  h_1(1,Y)+h_2(1,Y)(1-Y^2).$$
As above, $h_1(1,1) = 0$ implies that
$h_1(1,Y) = (Y-1)^2q(Y)$. Cancelling $Y-1$,
we get $0 = (Y-1)q(Y)+h_2(1,Y)(1+Y)$,
which implies that $h_2(1,1) = 0$, a contradiction.
\end{comment}

%----------------------------------------------------------------------------------------------------------------------------------------------
\section{Proof of Theorem \ref{t3m} } \label{proof}

Assertion (2) follows from assertion (1), by going to
the extension ring $\R[[\sqrt{x},y]]$, so it suffices to prove
(1). We can assume $f\ne 0$. We know $\R[[x,y]]$ is a UFD
\cite[Th. 6, p. 148]{ZS}. Factor $f$ into irreducibles in
$\R[[x,y]]$. Using the Preparation Theorem, we can assume the
factorization has the form
\[
f = ux^mg = ux^m\prod_{i=1}^{\ell} p_i^{m_i}
\]
where $u$ is a unit and each $p_i = p_i(y)$ is a monic polynomial in $y$ with coefficients in $\R[[x]]$,
with all coefficients except the leading coefficient in the maximal ideal of $\R[[x]]$.
We can reduce to the case where $m=0$ or $1$ and $g$ has no repeated irreducible factors.
Since $\pm u$ is a square in $\R[[x,y]]$, we can assume further that $u=\pm 1$.

\smallskip

Since $y$ and $x^{2n}-y$ are obviously in the preordering generated by $y$ and
$x^{2n}-y$, we can assume $y \nmid g$ and $y-x^{2n} \nmid g$.
More generally, if $g$ has an irreducible factor $p$ which has constant sign
on the set $y>0$ in the real spectrum  (see \cite{BCR}) of $\R((x,y))$ then,
by part (2) of Theorem \ref{t2m}, $\pm p$ is in the preordering generated by $y$.
Similarly, if $p$ has constant sign on the set $x^{2n}>y$ in the real spectrum
of $\R((x,y))$ then, by part (2) of Theorem \ref{t2m}
(using the fact that $\R[[x,y]] = \R[[x,x^{2n}-y]]$),
$\pm p$ is in the preordering generated by $x^{2n}-y$.
Consequently, we can assume that $g$ has no such irreducible factors.

\smallskip

Fix an irreducible factor $p$ of $g$ and consider the discrete
valuation on $\R((x,y))$ with associated valuation ring
$\R[[x,y]]_{(p)}$. The residue field is $L =
\text{qf}\frac{\R[[x,y]]}{(p)} = \frac{\R((x))[y]}{(p)}$
\cite[Th. 6, p. 148]{ZS}. Set $\overline{y} = y+(p)$. Since $p \ne
y$, $p \ne y-x^{2n}$, we know that $\overline{y} \ne 0$,
$\overline{y} \ne x^{2n}$. $L$ is a finite extension of the complete
discrete valued field $\R((x))$ so it either has no orderings
(if the residue field is $\Bbb{C}$) or two orderings (if the residue
field is $\R$).

\smallskip

Claim 1: $L$ has no ordering satisfying $0< \overline{y} <x^{2n}$.
Otherwise, pulling this ordering back to $\R((x,y))$, using Baer-Krull,
yields two orderings on $\R((x,y))$ satisfying $0<y<x^{2n}$, one with
$p>0$ and one with $p<0$. Since an irreducible factor $q$ of $f$ different
from $p$ has the same sign at each of these two orderings, and since $p$ has
multiplicity 1 in $f$, one of these two orderings must make $f <0$.
This contradicts our assumption and proves the claim.

\smallskip

Claim 2: $L$ has an ordering satisfying $\overline{y} > x^{2n}$ and also an
ordering satisfying $\overline{y}<0$. By assumption $p=p(y)$ is not always
positive on the set $y>0$ in the real spectrum of $\R((x,y))$, so there
exists an ordering of $\R((x,y))$, with real closure $R$ say, with $y>0$
and $ p(y) <0$, so the polynomial $p(t)$ (obtained by replacing $y$ by the new variable $t$) has a root $a >y$ in $R$.
Then $\overline{y} \mapsto a$ defines an $\R((x))$-embedding of $L$ into
$R$, so $L$ has an ordering satisfying $\overline{y} >0$, i.e.,
$\overline{y} > x^{2n}$. We prove the second assertion when $\deg(p)$ is odd.
The proof when $\deg(p)$ is even is similar. By assumption $p$ is not always
negative on the set $x^{2n}>y$ in the real spectrum of $\R((x,y))$,
so there exists an ordering of $\R((x,y))$ with real closure $R$ say,
with $y<x^{2n}$ and $p(y)>0$, so the polynomial $p(t)$ has a root $a<y$ in $R$.
Then $\overline{y} \mapsto a$ defines an $\R((x))$-embedding of $L$
into $R$, so $L$ has an ordering satisfying $\overline{y} <x^{2n}$,
i.e., $\overline{y} <0$.

\smallskip

Denote the valuation on $L$ by $v$. Since $p(\overline{y})=0$ we see that
$v(\overline{y})>0$. Since $L$ has an ordering satisfying
$\overline{y}>x^{2n}$, it follows that $v(\overline{y}) \le v(x^{2n})$.
At the same time, $v(\overline{y})=v(x^{2n})$ is not possible.
(If $v(\overline{y}) = v(x^{2n})$ then $\overline{y} = ux^{2n}$,
$u$ a unit. Since $\overline{y}$ is positive at one ordering and negative
at the other, the same would be true for $u$, which is not possible.)
Thus $0<v(\overline{y}) < v(x^{2n})$.

\smallskip

Of course, since the various roots $a$ of $p$ in the algebraic closure of
$\R((x))$ are conjugate to $\overline{y}$ over $\R((x))$,
they all have the same value $v(a) = v(\overline{y})$.

\smallskip

Write $f = \pm x^mp_1\dots p_{\ell}$ where the $p_i$ are irreducible,
$p_i = \sum_{j=0}^{k_i} b_{ij}y^j$, $b_{ik_i}=1$, $v(b_{i0}) = k_iv(a_i)$,
$v(b_{ij}) \ge (k_i-j)v(a_i)$, where $a_i$ is a fixed root of $p_i$.
We know $0<v(a_i)<v(x^{2n})$. Decompose $f$ as
\begin{equation}
\label{eq1}
f = f(0)+\sum_{\underline{j} \ne (0,\dots,0)}
\pm x^mb_{\underline{j}} y^{j_1+\dots+j_{\ell}}
\end{equation}
where $\underline{j} := (j_1,\dots,j_{\ell})$,
$b_{\underline{j}}:= b_{1j_1}\dots b_{\ell j_{\ell}}$
and $f(0) := \pm x^m b_{10}\dots b_{\ell 0}$.

\smallskip

Claim 3: $f(0)$ is positive at both orderings of $\R((x))$, i.e.,
$f(0)$ is a square in $\R[[x]]$.  Suppose to the contrary that
$f(0)$ is negative at one of the orderings of $\R((x))$.
Consider the discrete valuation on $\R((x,y))$ with valuation ring
$\R[[x,y]]_{(y)}$ and residue field $\R((x))$. Pulling the culprit
ordering of $\R((x))$ back to $\R((x,y))$, using Baer-Krull,
yields two orderings of $\R((x,y))$, one of which satisfies
$x^{2n}>y>0$ and $f<0$. This is a contradiction.

\smallskip

We write each term $\pm x^mb_{\underline{j}}y^{j_1+\dots+j_{\ell}}$,
$\underline{j} \ne (0,\dots,0)$ in (\ref{eq1})
as $$(c_{\underline{j}}\pm x^mb_{\underline{j}})y^{j_1+\dots+j_{\ell}}
+ c_{\underline{j}}(x^{2n(j_1+\dots+j_{\ell})}-y^{j_1+\dots+j_{\ell}})-
c_{\underline{j}}x^{2n(j_1+\dots+j_{\ell})}.$$ Factoring in the obvious way,
we see that $x^{2n(j_1+\dots+j_{\ell})}-y^{j_1+\dots+j_{\ell}}$ lies in
the preordering generated by $x^{2n}-y$ and $y$.
To complete the proof, it suffices to show we can choose the elements
$c_{\underline{j}} \in \R[[x]]$, $\underline{j} \ne (0,\dots,0)$,
so that $$\ c_{\underline{j}}\pm x^mb_{\underline{j}}, \ c_{\underline{j}} \
\text{ and } \ f(0)-\sum_{\underline{j} \ne (0,\dots,0)} c_{\underline{j}}x^{2n(j_1+\dots+j_{\ell})}$$
are squares in $\R[[x]]$. Since $\underline{j} \ne (0,\dots,0)$,
\begin{align*} v(x^mb_{\underline{j}}) = &\ v(x^m)+\sum_i v(b_{ij_i}) \\
\ge &\ v(x^m)+\sum_i (k_i-j_i)v(a_i) \\ = &\ v(x^m)+\sum_i k_iv(a_i)-
\sum_i j_iv(a_i) \\ > &\ v(x^m)+\sum_i k_iv(a_i)-\sum_i j_iv(x^{2n})
\\ = &\ v(x^m)+\sum_i v(b_{i0})- \sum_i j_iv(x^{2n}) \\
= &\ v(\frac{f(0)}{x^{2n(j_1+\dots+j_{\ell})}}).
\end{align*}
We choose the $c_{\underline{j}}$ as follows:
If $\underline{j} \ne (k_1,\dots,k_{\ell})$ or $\underline{j}=(k_1,\dots,k_{\ell})$ and $m=1$,
then $x^mb_{\underline{j}}$ has positive value. In this case, we choose $c_{\underline{j}}$
with small positive lowest coefficient and with
$$v(c_{\underline{j}})= \max\{v(\frac{f(0)}{x^{2n(j_1+\dots+j_{\ell})}}), 0\}.$$
In the remaining case, where $m=0$ and $\underline{j}
= (k_1,\dots,k_{\ell})$, $b_{ij_i} = 1$, $i=1,\dots, \ell$,
and we choose $c_{\underline{j}} = 1$.
The point is, with this choice of $c_{\underline{j}}$,
for each $\underline{j} \ne (0,\dots,0)$,
either $c_{\underline{j}}x^{2n(j_1+\dots+j_{\ell})}$ has larger value than $f(0)$ or,
it has the same value as $f(0)$, but its lowest coefficient is small.

\section{Concluding Remarks}
\label{ljubljana}

1. Theorems \ref{t2m} and \ref{t3m} do not cover all interesting cases. The general question remains: When is a finitely generated preordering of $\R[[x,y]]$ saturated? Recall that the \textbf{saturation} of a preordering $T$ of a general ring $A$ (commutative with 1) is the intersection of all orderings of $A$ containing $T$, and that $T$ is said to be \textbf{saturated} if it coincides with its saturation.
% Recall that a preordering $T$ of a general ring $A$ (commutative with 1) is said to be \textbf{saturated} if $\forall f \in A$, $f\ge 0$ at each ordering of $A$ such that $t\ge 0$ $\forall$ $t\in T$ $\Rightarrow$ $f\in T$.

2. The following preorderings of $\R[[x,y]]$ are saturated:
\begin{enumerate}
\item[(i)] The preordering of $\R[[x,y]]$ generated by $y$ and $y-x^n$, $n$ odd, $n\ge 3$.

\item[(ii)] The preordering of $\R[[x,y]]$ generated by $y^2-x^n$, $n$ odd, $n\ge 3$.
\end{enumerate}
 Saturation in case (i) is a consequence of saturation in case (ii), by going to the extension ring $\R[[x,\sqrt{y}]]$. In an analogous way, saturation in case (ii) is a consequence of \cite[Th. 5.1]{f}, by going to the extension ring $$A := \frac{\R[[x,y]][z]}{(z^2-y^2+x^n)} = \frac{\R[[x,y,z]]}{(z^2-y^2+x^n)}.$$ \cite[Th. 5.1]{f} asserts that the ring $A$ defined above satisfies psd = sos, i.e., that the preordering of $A$ consisting of sums of squares in saturated.\footnotemark\footnotetext{ The authors wish to thank the referee for bringing this result to their attention, and pointing out its application to cases (i) and (ii).} Actually, \cite[Th. 5.1]{f} is stated in terms of 
analytic function germs. What we are quoting here is the formal power series version of the result.  Note: Knowing saturation holds in case (i) allows one to extend Corollary \ref{c7m}, adding an additional case to the list.
%\begin{enumerate}
%\item[(5)] there exists $i,j$ such that $p$ is a non-singular zero of $g_i$ and $g_j$,
%$g_i$ and $g_j$ share a common tangent at $p$ but do not cross each other
%at $p$, and $K$ is described locally at $p$ as the region between $g_i =0$
%and $g_j=0$; or
%\end{enumerate}

3. It is still not known if the following preorderings of $\R[[x,y]]$ are saturated:
\begin{enumerate}
\item[(iii)] The preordering of $\R[[x,y]]$ generated by $y,y-x^n$ and $x^m-y$, $n$ odd, $m$ even, $n>m \ge 2$.

\item[(iv)] The preordering of $\R[[x,y]]$ generated by $y, y-x^n, x^m-y$ and $x^m(1+a(x))-y$, $n$ odd, $m$ even, $n>m \ge 2$, $a(x) \in \R[[x]]$, $a(0)=0$.
\end{enumerate}
A positive answer in cases (iii) and (iv), coupled with what we already know by Theorems \ref{t2m} and \ref{t3m} and case (i) above, would complete our understanding of saturation for preorderings of $\R[[x,y]]$ generated by finitely many elements of order $\le 1$. The proof of this assertion will not be given here. The \textbf{order} of $f\in \R[[x,y]]$ is defined to be the greatest integer $k\ge 0$ such that $f \in \frak{m}^k$, where $\frak{m}$ denotes the maximal ideal of $\R[[x,y]]$.

4. The case where some of the generators have order $\ge 2$ seems to be pretty much wide open. Case (ii) is of this type, as is the example given earlier, in Remark \ref{rm5} (ii). If $g \in \R[x,y]$ and psd = sos holds for the ring $A = \frac{\R[[x,y]][z]}{(z^2-g)}$, then the preordering of $\R[[x,y]]$ generated by $g$ is saturated. Combining this with \cite[Th. 3.1]{f} yields a variety of examples of this sort where $g$ has order 2 or 3 and saturation holds.

\bn
\bn
\adresse
\end{document}